\documentclass{amsart}
\usepackage{amsmath,amsfonts,amssymb, enumerate}

\newcommand{\R}{{\mathbb R}}
\theoremstyle{plain}

\newtheorem{thm}{Theorem}[section]
\newtheorem{lem}[thm]{Lemma}

\theoremstyle{definition}
\newtheorem{definition}{Definition}[section]
\theoremstyle{remark}
\newtheorem{remark}{Remark}[section]

\newtheorem{Cor}{Corollary}[section]
\theoremstyle{remark}

\numberwithin{equation}{section}

\begin{document}
\title[General and alien solutions of a functional equation...]
{General and alien solutions of a functional equation and of a functional inequality}

\author[W.~Fechner]{W\l{}odzimierz~Fechner}
\address{W\l{}odzimierz Fechner, 
Institute of Mathematics, 
University of Silesia, 
Bankowa 14, 
40-007 Katowice, 
POLAND}
\email{wlodzimierz.fechner@us.edu.pl; fechner@math.us.edu.pl}

\author[E.~Gselmann]{Eszter Gselmann}
\address{Eszter Gselmann, 
Institute of Mathematics, 
University of Debrecen, 
P.~O.~Box: 12., 
H--4010, 
Debrecen,
HUNGARY}
\email{gselmann@science.unideb.hu}

\thanks{The research of the second author has been supported by the Hungarian Scientific Research Fund
(OTKA) Grant NK 81402 and by the T\'AMOP 4.2.1./B-09/1/KONV-2010-0007 project implemented through the New Hungary Development Plan co-financed by the European Social Fund and the European Regional Development Fund}
\subjclass[2000]{39B52, 39B62, 39B72}
\keywords{functional equation, functional inequality, additive function, integral domain}

\begin{abstract}
The purpose of the present paper is to solve 
(under some assumption on the domain) the equation 
$$
g(x+y)-g(x)-g(y)=xf(y)+yf(x). 
$$
After determining the general solutions, we will investigate the 
so--called alien solutions. 
Finally, we will discuss the real solutions of the following related functional inequality:
$$
g(x+y)-g(x)-g(y)\geq xf(y)+yf(x). 
$$
\end{abstract}

\maketitle

\section{Introduction}

In mathematics there exist several notions concerning functions 
that are defined through two or more identities. 
For example, if $P$ and $Q$ are rings, then the function 
$f:P\to Q$ is termed a \emph{homomorphism} between $P$ and $Q$ if it is additive and multiplicative, i.e. if
\begin{equation}
f(x+y)=f(x)+f(y) \qquad  \left(x, y\in P\right)
	\label{A}
\end{equation}
and 
\begin{equation}
f(xy)=f(x)f(y) \qquad \left(x, y\in P\right)
	\label{M}
\end{equation}

Another example is, for instance, the notion of derivations. 
Let $Q$ be a ring and $P$ be its subring. 
A function $f:P\to Q$ is called a \emph{derivation} if it is additive and 
$$
f(xy)=xf(y)+f(x)y
\qquad 
\left(x, y\in P\right). 
$$

The following question naturally arises: Is it possible to characterize 
such type of functions via a single equation?
This problem was firstly investigated by \textsc{J.~Dhombres} in 
\cite{Dho88}. He examined the equation 
$$
af(xy)+bf(x)f(y)+cf(x+y)+d\left(f(x)+f(y)\right)=0, 
$$
where the unknown function $f$ maps a ring $X$ to a field
$Y$ and the constants $a, b, c$ and $d$ belong to the center of $Y$. 

Ten years later, in 1998 \textsc{R.~Ger} succeed to strengthen the results 
of \cite{Dho88}. 
In the paper \cite{Ger98} he proved several statements 
concerning the following equation which is the sum of \eqref{A} and \eqref{M}:
$$
f(x+y)+f(xy)=f(x)+f(y)+f(x)f(y). 
$$
In this direction some further results can be found in 
\textsc{Ger} \cite{Ger00} and in \textsc{Ger--Reich} \cite{GR10}. 

Similarly to the notion of homomorphisms, derivations 
can be characterized analogously. 
For example, in \cite{Gse09} the functional equation 
$$
f(x+y)-f(x)-f(y)=g(xy)-xg(y)-yg(x)
$$
is solved under the assumption that the domain of the 
functions $f$ and $g$ is a commutative field and the range of 
these functions is a vector space over this field. 

In parallel, several authors discussed various versions of the following functional inequality:
$$g(x+y) - g(x) - g(y) \geq \phi(x,y),$$ with some additional assumptions upon $g$ and $\phi$ 
(see \textsc{Baron--Kominek} \cite{BK}, \textsc{Cho\-czewski--Girgensohn--Kominek} \cite{C}, 
\textsc{Renardy} \cite{R}, see also \cite{W1, W2, W3}).
In Section 4 as a special case of $\phi$ we take $\phi(x,y) = f(x)y + yf(x)$ with $f$ satisfying certain further conditions. 

In this paper we will continue the above--mentioned research and we will examine the functional equation 
\begin{equation}
 \tag{$\ast$}
g(x+y)-g(x)-g(y)=xf(y)+yf(x), 
\end{equation}
where the unknown functions $f$ and $g$ are defined on an integral domain. 
Firstly, we will find the general solution of equation $\left(\ast\right)$. 
After that, we will study the following problem. 
Let $X$ be a ring. It is obvious that in case the function 
$g:X\to X$ is additive and the function $f:X\to X$ fulfills 
\begin{equation}
xf(y)+yf(x)=0 \qquad  \left(x, y\in X\right), 
	\label{d1}
\end{equation}
then equation $\left(\ast\right)$ holds. 
These solutions are the so--called \emph{alien solutions} of the 
equation in question. We will point out that equation $\left(\ast\right)$ 
has solutions that are not alien (in the above sense). 
Moreover, we will also give necessary and sufficient conditions on 
the functions $f$ and $g$ to be alien solutions of the equation in
question. 

In last section we confine ourselves to the following functional inequality 
\begin{equation}
 \tag{$\ast\ast$}
g(x+y)-g(x)-g(y)\geq xf(y)+yf(x), 
\end{equation}
where the unknown functions $f, g\colon \mathbb{R}\to \mathbb{R}$ satisfy some additional technical assumptions.

Finally, let us mention that functional equations, similar to equation 
$\left(\ast\right)$ were considered by several authors.  
For example in \textsc{Ebanks--Kannappan--Sahoo} \cite{EKS92}
the authors characterize all functions 
$f\colon \mathbb{K}\to G$ for which $f(x+y)-f(x)-f(y)$ depends only on the product $xy$ for all $x,y\in \mathbb{K}$, 
where $\mathbb{K}$ is a commutative field and $G$ is a uniquely $q$-divisible abelian group. 
In \textsc{Ebanks} \cite{Eba08} the equation 
$$
f(x+y)-f(x)-f(y)=g\left(H(x, y)\right)
$$
is investigated, where the unknown functions 
$f, g$ defined on a nonvoid interval $I\subset \mathbb{R}$  and $H\left(I\times I\right)$, respectively, satisfy some mild 
regularity conditions and the given function $H$ fulfills some stronger regularity assumption. 
Furthermore, we also note that in \textsc{J\'{a}rai--Maksa--P\'{a}les} \cite{JMP04} 
the authors described all Cauchy differences that can be written as a quasisum, i.e. they have dealt with 
the functional equation 
$$
f(x+y)-f(x)-f(y)=\alpha\left(\beta(x)+\beta(y)\right)
$$
for the unknown function $f:I\rightarrow\mathbb{R}$, 
and it is solved under the supposition that the functions $\alpha$ and $\beta$ are strictly 
monotonic. 

Let us emphasize that in our Theorem \ref{T} below \emph{no regularity assumption} 
is involved. Furthermore, we will work in a quite general framework concerning the domain and the target space of the unknown functions. 

\section{Preliminaries}

In this section we will fix the notations and the 
terminology that will be used subsequently. 
We refer the reader to the monographs of \textsc{Kuczma} \cite{Kuc85} and of \textsc{Shafarevich} \cite{Sha90}. 


\begin{definition}
By an \emph{integral domain} we understand a commutative unitary ring that contains no zero 
divisors.
\end{definition} 

The following notion will also be used in the next section. 

\begin{definition}
Let $n$ be a positive integer and G  an abelian group. 
An element $x\in G$ is said to be \emph{divisible by $n$}  if there is $y\in G$ such that $x=ny$. 
\end{definition}

\begin{lem}\label{L}
Let $X$ be an integral domain and assume that the function 
$f:X\to X$ fulfills \eqref{d1}. Then the function $2f:X\to X$ is identically zero. 
\end{lem}
\begin{proof}
First let us substitute $x\to 1$ and $y\to 1$ to derive $2f(1)=0$. 
Further, with the substitution $y\to 1$, the above equation yields that 
$$
2xf(1)+2f(x)=0
$$
is satisfied for all $x\in X$. Due to $2f(1)=0$, we obtain that 
$2f$ is identically zero on $X$, as claimed. 
\end{proof}

\begin{remark}
In general it is not true that $f=0$ in the foregoing lemma (however, under additional assumption  of the divisibility by $2$ postulated only for a single element $f(1)$ one can easily obtain $f=0$). 
If one take $X=\mathbb{Z}_2$ and consider the mapping   $f_1(x) = x$ then it is easy to check that
this functions provides (nonzero) solutions of the equation. On the other hand, the maps $f_2(x) = 1$ and $f_3(x) = x+1$ (the  remaining nonzero self-mappings on $X$) do not solve it.
\end{remark}

Let us also mention the following easily to verify result (the converse implication of a theorem due to \textsc{Jessen--Karpf--Thorup} \cite[Theorem 2]{JKT69}). 

\begin{thm}\label{JKT}
Let $X$ be an Abelian group and $f:X\to X$ an arbitrary function. Then the function $F:X\times X \to X$ defined by 
$$
F(x, y)=f(x+y)-f(x)-f(x) 
\qquad 
\left(x, y\in X\right)
$$ 
is symmetric, i.e., 
$$
F(x, y)=F(y, x)
\qquad 
\left(x, y\in X\right)
$$
and fulfills the co--cycle equation, that is, 
$$
F(x+y, z)+F(x, y)=F(x, y+z)+F(y, z)
$$
holds for all $x, y, z\in X$. 
\end{thm}

Finally, we will need the following two results. Recall that a map $f\colon X \to \R$ defined on an Abelian goup $X$ is subadditive if $$f(x+y) \leq f(x) + f(y)$$ for all $x, y\in X$.

\begin{Cor}[\cite{W1}, Corollary 1]\label{c1}
Assume that $X$ is an Abelian group, $f\colon X \rightarrow \mathbb{R}$ and $\phi\colon X \times
X\rightarrow \mathbb{R}$ satisfy 
\begin{equation}\label{main}
f(x+y) - f(x) - f(y) \geq \phi(x , y) \quad (x, y \in X),
\end{equation}
\begin{equation}\label{NP}
\phi (x,-y) \geq -\phi (x, y)  \quad (x, y \in X),
\end{equation}
\begin{equation}\label{2}
\left\{ \begin{array}{l} \limsup _{n \rightarrow + \infty},
\frac{1}{4^{n}} \phi(2^{n}x,2^{n}x) < + \infty \quad (x \in X),
\\
 \liminf _{n \rightarrow +
\infty} \frac{1}{4^{n}} \phi(2^{n}x,2^{n}y) \geq \phi(x,y)
\quad (x, y \in X)
\end{array}\right.
\end{equation}
and
\begin{equation}\label{--}
\phi(-x,-y) = \phi(x,y) \quad (x, y \in X).
\end{equation}
Then there exists a subadditive function $A\colon X \rightarrow\mathbb{R}$
such that 
$$f(x) = \frac{1}{2} \phi(x, x ) - A(x)  \quad 
(x \in X ).
$$ 
Moreover, $\phi$ is biadditive and symmetric.
\end{Cor}

\begin{Cor}[\cite{W2}, Corollary 8]\label{c2} Assume $X$ to be uniquely $2$-divisible Abelian group and that
$f\colon X \rightarrow \mathbb{R}$, $\phi\colon X \times X\rightarrow \mathbb{R}$
satisfy \eqref{main}, \eqref{NP}, 
\begin{equation}\label{2W}
\phi(2x,2x) \leq 4 \phi(x,x) \quad (x \in X)
\end{equation} jointly with
\begin{equation}\label{0WE}
\forall_{x \in X} \exists_{k_{0} \in \mathbb{N}} \forall_{k \geq k_{0}}
f\left(\frac{x}{2^{k}}\right) + f\left(-\frac{x}{2^{k}}\right)  \geq 0 .
\end{equation}
Then there exists an additive function $a\colon X \rightarrow \mathbb{R}$
such that $$f(x) = \frac{1}{2} \phi(x, x ) + a(x) \quad (x\in X).$$
Moreover, $\phi$ is biadditive and symmetric.
\end{Cor}

\section{Functional equation $\left(\ast\right)$}

The main result in this section is the following:

\begin{thm}\label{T}
Let $X$ be an integral domain. Then the functions $f, g \colon X \to X$ fulfill functional equation $\left(\ast\right)$
for all $x, y\in X$,
if and only if, there exist two mappings $A_1, A_2 \colon X \to X$ and a constant $c \in X$ such that
$A_1$ and $2A_2$ are additive and
\begin{align}\label{af}
4f(x) &= 2A_1(x) + 2cx^2 \qquad \left(x\in X\right), \\
6g(x) &= A_2(x) + 3xA_1(x) + cx^3   \qquad \left(x\in X\right).\label{ag}
\end{align}
\end{thm}
\begin{proof}
The \emph{if} part is a straightforward computation and therefore we will confine ourselves to the \emph{only if} part.

First, observe that substitution $y\to 0$ shows that $-g(0) = xf(0)$ for each $x \in X$ which easily implies that 
$$
f(0)=g(0)=0.
$$

Next, apply equation $\left(\ast\right)$ with $y\to x$ to obtain
\begin{equation}\label{aux1}
g(2x) - 2g(x) = 2xf(x) \quad (x \in X).
\end{equation}
Now, let us define four new functions $f_o, f_e, g_o, g_e\colon X \to X$ by the following formulas:
\begin{align*}
f_o(x) &= {f(x) - f(-x)}, \qquad f_e(x) = {f(x) + f(-x)}  \quad \left(x \in X\right);\\ 
g_o(x) &= {g(x) - g(-x)}, \qquad g_e(x) = {g(x) + g(-x)} \quad \left(x \in X\right).
\end{align*}
Replace in $\left(\ast\right)$ $x$ by $-x$ and $y$ by $-y$, respectively, to arrive at 
$$
g(-x-y) - g(-x) - g(-y) = -xf(-y) - yf(-x) \qquad \left(x, y \in X\right).
$$
By adding and subtracting this equality and $\left(\ast\right)$ side-by-side we deduce the following two equalities:
\begin{align}
g_e(x+y) - g_e(x) - g_e(y) &= xf_o(y) + yf_o(x) \qquad \left(x, y \in X\right);\label{eo}\\
g_o(x+y) - g_o(x) - g_o(y) &= xf_e(y) + yf_e(x) \qquad \left(x, y \in X\right).\label{oe}
\end{align}
On the other hand, substitution $y\to-x$ in $\left(\ast\right)$ leads to 
\begin{equation}\label{gf}
g_e(x) = xf_o(x) \qquad \left(x \in X\right).
\end{equation}
Further, substitution $x\to 2x$ and $y\to-x$ together with \eqref{aux1} gives us the equality 
$$
g_e(x) +2xf_e(x) = xf(2x) \qquad \left(x \in X\right).
$$
This and identity \eqref{gf} prove that
\begin{equation}\label{31}
f(2x) = 3f(x) + f(-x) \qquad \left(x \in X\right).
\end{equation}
Further, this implies the following properties of the functions $f_o$ and $f_e$:
\begin{equation}\label{24}
f_o(2x) = 2f_o(x)  \qquad \text{and} \qquad f_e(2x) = 4f_e(x) \qquad \left(x \in X\right).
\end{equation}

Now, join \eqref{eo} with \eqref{gf} to deduce
$$(x+y)f_o(x+y) =(x+y) [f_o(x)+f_o(y)] \qquad \left(x, y \in X\right),$$
which together with the fact that $X$ is an integral domain, imply that $f_o$ is additive. 
Thus there exists an additive function $A_{1}: X\to X$ such that 
$$
f_{o}(x)=A_{1}(x) \qquad 
\left(x\in X\right). 
$$
Additionally, using \eqref{gf} we also get that 
$$
 g_e(x)=xA_{1}(x) \qquad 
\left(x\in X\right). 
$$


It remains to solve equation \eqref{oe}. For our convenience let us denote the Cauchy difference of $g_o$ by $C$, that is, let
$$
C(x,y) = g_o(x+y) - g_o(x) - g_o(y) \qquad \left(x, y \in X\right).
$$ 
Due to Theorem \ref{JKT} the function $C$ fulfills the co--cycle equation 
$$
C(x+y,z) + C(x,y) = C(x,y+z) + C(y,z) \qquad \left(x, y, z \in X\right).
$$ 
Comparing this with the right hand side of \eqref{oe}, after some rearrangements we arrive at 
$$
x[f_e(y+z) -f_e(y)-f_e(z)]=z[f_e(x+y) - f_e(x) - f_e(y)] \qquad \left(x, y, z \in X\right).
$$
Apply this for for $z\to y$ and use the second equality from \eqref{24} to deduce the following relation
$$
2xf_e(y) = y[f_e(x+y) -f_e(x)-f_e(y)] \quad \left(x, y \in X\right).
$$ 
If we multiply both sides of the foregoing formula by $x$, then the following equality:
$$
2x^{2}f_e(y) = xy[f_e(x+y) -f_e(x)-f_e(y)] \quad \left(x, y \in X\right)
$$  
can be derived. 
Let us observe that the right hand side of this equation is symmetric in $x$ and $y$. 
Therefore, so is the left hand side. This implies however that 
$$
2x^2f_e(y) = 2y^2f_e(x)
$$ 
hold for any $x\in X$. 
If we substitute $y\to 1$ then we see that 
$$2f_e(x)=2cx^2
$$ 
for each $x \in X$, where $c = f_e(1)$.

To finish the proof we need to determine the function $g_o$. In view of the above representation 
of the function $f_e$, equation \eqref{oe} turns into
\begin{equation}\label{go}
2[g_o(x+y) - g_o(x) - g_o(y)] = 2cxy(x+y) \qquad \left(x, y \in X\right).
\end{equation}
Define the function $A_2: X\to X$ through the formula
$$
A_2(x)=3g_o(x)-c x^{3} \qquad 
\left(x\in X\right), 
$$
(the constant $c$ is the same as above). 
A direct calculation shows that in this case equation \eqref{go} yields that the function 
$2A_{2}$ is additive. Therefore
$$
3g_{o}(x)=A_{2}(x)+cx^{3} 
\qquad 
\left(x\in X\right). 
$$

To conclude the proof it suffices to use the above results concerning the functions 
$f_{o}, f_{e}, g_{o}$ and $g_{e}$ jointly with the fact that 
$$
2f(x)=f_{e}(x)+f_{o}(x) \qquad 
\text{and}
\qquad 
2g(x)=g_{e}(x)+g_{o}(x) 
\qquad 
\left(x\in X\right). 
$$
\end{proof}

If we assume additionally that the ring $X$ appearing in Theorem \ref{T} is uniquely divisible by $2$ and $3$ then formulas \eqref{af} and \eqref{ag} can be simplified. We have the following corollary.

\begin{Cor}\label{C1}
Let $X$ be an integral domain which is uniquely divisible by $2$ and $3$ and assume that equation $\left(\ast\right)$ holds for $f, g \colon X \to X$. Then, and only then, there exist two additive mappings $A_1, A_2 \colon X \to X$ and a constant $c \in X$ such that
\begin{align}\label{afc}
f(x) &= A_1(x) + cx^2 \qquad \left(x\in X\right), \\
g(x) &= A_2(x) + xA_1(x) + \frac13cx^3   \qquad \left(x\in X\right).\label{agc}
\end{align}
\end{Cor}

\begin{remark}
We are grateful to Professor Maciej Sablik for a remark that the foregoing Corollary can be deuced from more general lemmas from papers M. Sablik \cite[Lemma 2.3]{S} and A. Lisak, M. Sablik \cite[Lemma 1]{LS}. More precisely, these general results imply that each solution of equation $\left(\ast\right)$ is a polynomial function of some degree. What remains to be done is to calculate the exact form of this polynomial function. However, unlike to our Theorem \ref{T} both \cite[Lemma 2.3]{S} and \cite[Lemma 1]{LS} require unique divisibility of the target space.
\end{remark}

Making use of Corollary \ref{C1}, we can easily derive a criteria on the 
functions $f$ and $g$ to be the alien solutions of equation $\left(\ast\right)$. 
By Lemma \ref{L} if $f$ and $g$ are alien then $f=0$ and $g$ is additive.

\begin{Cor}
 Let $X$ be an integral domain which is uniquely divisible by $2$ and $3$ and 
consider the functions $f, g:X\to X$ and assume that equation 
$\left(\ast\right)$  holds. 
Then the following statements are equivalent: 
\begin{enumerate}[(i)]
\item $f=0$ and $g$ is additive;
\item the function $f$ is even and $f(1)=0$;
\item the function $g$ is odd and $g(2)=2g(1)$. 
\end{enumerate}
\end{Cor}

\begin{proof}
The proof is a direct calculation based on Corollary \ref{C1}.
\end{proof}

Finally, we investigate the case when the functions occurring in equation 
$\left(\ast\right)$ are the same. In this case we prove the following. 

\begin{Cor}
 Let $X$ be an integral domain which is uniquely divisible by $2$ and $3$. 
Assume that the function $f:X \to  X$ fulfills 
$$
f(x+y)-f(x)-f(y)=xf(y)+yf(x)
$$
for any $x, y\in X$. Then and only then, the function $f$ is identically zero. 
\end{Cor}
\begin{proof}
Follows immediately from Corollary \ref{C1}.
\end{proof}

\section{Functional inequality $\left(\ast\ast\right)$}

We will apply Corollaries \ref{c1} and \ref{c2} to obtain two analogues of Theorem \ref{T} for inequality $\left(\ast\ast\right)$ under some additional assumptions.

\begin{thm}\label{p}
Assume that the functions $f, g:\mathbb{R} \to  \mathbb{R}$ 
fulfill inequality $\left(\ast\ast\right)$ for each $x, y \in \mathbb{R}$. 
If $f$ is odd and $f(2x) = 2f(x)$ for each $x \in \mathbb{R}$ then $f$ is 
additive and there exists a subadditive mapping $A\colon \mathbb{R} \to \mathbb{R}$ such that 
$$g(x) = xf(x) - A(x), $$ for all $x \in \mathbb{R}$.
\end{thm}
\begin{proof}
Let us define $$\phi(x,y) = xf(y) + yf(x) \qquad 
\left(x, y\in \mathbb{R}\right).$$ 
One may calculate that thanks to our assumptions upon $f$ we have $$\phi(x,-y) =\phi(-x,y) =-\phi(x,y)\qquad \left(x, y\in \mathbb{R}\right)$$ 
and $$\phi(2x,2y) = 4\phi(x,y)\qquad \left(x, y\in \mathbb{R}\right).$$ 
This implies that assumptions of Corollary \ref{c1} are satisfied by $g$ and $\phi$. 
Therefore, we obtain the existence of a subadditive mapping $A\colon \mathbb{R} \to \mathbb{R}$ such that 
$$g(x) = \frac{1}{2}\phi(x,x) - A(x) \qquad \left(x\in \mathbb{R}\right)$$ and additionally we get that $\phi$ is biadditive. 
Using the latter assertion we may calculate that
\begin{align*}
   xf(y+z) + (y+z)f(x) &=  \phi(x,y+z) = \phi(x,y) + \phi(x,z) \\&= xf(y) + yf(x) + xf(z) + zf(x) 
\end{align*} 
 for all $x, y \in \mathbb{R}$ and this applied for $x=1$ gives us the additivity of $f$. 
To finish the proof note that $\phi(x,x) = 2xf(x)$  for all $x \in \mathbb{R}$.
\end{proof}


\begin{thm}
Under assumptions of Theorem \ref{p}, if additionally for each $x\in \mathbb{R}$ there exists $k_0\in \mathbb{N}$ 
such that for every $k \geq k_0$ we have $$g\left(\frac{x}{2^k}\right)+  g\left(-\frac{x}{2^k}\right)\geq 0,$$ 
then the map $A\colon \mathbb{R} \to \mathbb{R}$ postulated by Theorem \ref{p} is additive.
\end{thm}
\begin{proof}
Preserving notations and using some calculations from the previous proof one can check 
that all assumptions of Corollary \ref{c2} are satisfied by $g$ and $\phi$ and the assertion follows from this result.
\end{proof}

\begin{remark}
One may easily observe that the alienation effect for inequality $\left(\ast\ast\right)$ does not hold 
under assumptions of the foregoing two theorems, except in the trivial case $f=0$. 
Indeed, assume that assertion of Theorem \ref{p} holds. 
To get the alienation effect we expect that  
$$g(x+y) - g(x) - g(y) \geq 0$$ and $$xf(y) + yf(x) \leq 0$$ for each $x, y \in \mathbb{R}$. 
The second inequality applied for $y=1$ implies that $$f(x) \leq -f(1)x \qquad (x \in \mathbb{R})$$ 
and this easily gives us that $f(x)=-f(1)x$ for each $x \in \mathbb{R}$ and consequently $f(1)=0$ and thus $f=0$.  
\end{remark}

\end{document}